\documentclass[10pt]{article}
\usepackage{indentfirst, latexsym, amsmath,bm,amssymb}
\usepackage{tikz}
\usepackage{color,soul}
\usepackage{amsmath}
\usepackage{amssymb}
\usepackage{mathrsfs}
\usepackage{mathtools}
\usepackage{stmaryrd}
\usepackage{bbold}

\begin{document}

\newtheorem{theorem}{Theorem}[section]
\newtheorem{corollary}[theorem]{Corollary}
\newtheorem{definition}[theorem]{Definition}
\newtheorem{conjecture}[theorem]{Conjecture}
\newtheorem{question}[theorem]{Question}
\newtheorem{lemma}[theorem]{Lemma}
\newtheorem{remark}[theorem]{Remark}
\newtheorem{proposition}[theorem]{Proposition}
\newtheorem{example}[theorem]{Example}
\newenvironment{proof}{\noindent {\bf
Proof.}}{\rule{3mm}{3mm}\par\medskip}
\newcommand{\pp}{{\it p.}}
\newcommand{\de}{\em}

\newcommand{\JEC}{{\it Europ. J. Combinatorics},  }
\newcommand{\JCTB}{{\it J. Combin. Theory Ser. B.}, }
\newcommand{\JCT}{{\it J. Combin. Theory}, }
\newcommand{\JGT}{{\it J. Graph Theory}, }
\newcommand{\ComHung}{{\it Combinatorica}, }
\newcommand{\DM}{{\it Discrete Math.}, }
\newcommand{\ARS}{{\it Ars Combin.}, }
\newcommand{\SIAMDM}{{\it SIAM J. Discrete Math.}, }
\newcommand{\SIAMADM}{{\it SIAM J. Algebraic Discrete Methods}, }
\newcommand{\SIAMC}{{\it SIAM J. Comput.}, }
\newcommand{\ConAMS}{{\it Contemp. Math. AMS}, }
\newcommand{\TransAMS}{{\it Trans. Amer. Math. Soc.}, }
\newcommand{\AnDM}{{\it Ann. Discrete Math.}, }
\newcommand{\NBS}{{\it J. Res. Nat. Bur. Standards} {\rm B}, }
\newcommand{\ConNum}{{\it Congr. Numer.}, }
\newcommand{\CJM}{{\it Canad. J. Math.}, }
\newcommand{\JLMS}{{\it J. London Math. Soc.}, }
\newcommand{\PLMS}{{\it Proc. London Math. Soc.}, }
\newcommand{\PAMS}{{\it Proc. Amer. Math. Soc.}, }
\newcommand{\JCMCC}{{\it J. Combin. Math. Combin. Comput.}, }
\newcommand{\GC}{{\it Graphs Combin.}, }
\title{The Distance Energy of  Clique Trees\thanks{ This work is Supported by the National Natural Science Foundation of China (Nos. 11601337, 11971311,11531001 and 11971319), the Foundation of Shanghai Normal University (No. SK201602) and the Foundation of Shanghai Municipal Education Commission,  and the Montenegrin-Chinese Science and Technology Cooperation Project (No.3-12).}}
\author{
Ya-Lei Jin$^1$, Rui Gu$^1$, Xiao-Dong Zhang$^2$ \\
$^1$Department of Mathematics,
Shanghai Normal University, \\
 100 Guilin road, Shanghai, 200234, P.R. China
\\
$^2$School of Mathematical Sciences, MOE-LSC, SHL-MAC,\\ Shanghai Jiao Tong University,
\\ 800 Dongchuan road, Shanghai, 200240,  P.R. China.
\\
Email: yaleijin@shnu.edu.cn, gurui259011@163.com, xiaodong@sjtu.edu.cn
 }


\maketitle

\begin{abstract}  The distance energy of a simple connected graph $G$ is  defined as the sum of absolute values of its  distance eigenvalues. In this paper, we mainly give a positive answer to a conjecture of distance energy of clique trees proposed by Lin, Liu and Lu [H.~Q.~ Lin, R.~F.~Liu, X.~W.~Lu, The inertia and energy of the distance matrix of a connected graph, {\it Linear Algebra Appl.,} 467 (2015), 29-39.].
\end{abstract}



\section{Introduction}

  Graham and Pollak  \cite{graham1971} in 1971 discovered an interesting and elegant result on the distance matrix of a tree that  the determinant of distance matrix of  any $n-$vertex tree   is $(-1)^{n-1}(n-1)2^{n-2}$, which is independent with the structure of  trees. Furthermore, Graham  and Lov\'{a}sz in 1978 \cite{graham1978}  derived  the coefficients of the characteristic polynomial of the distance matrix of a graphs  in term of  as certain tied linear combinations of the numbers of various subgraphs of a graph.  Recently,  Cheng and Lin \cite{cheng2018} presented a class of graph whose distance determinant  are independent of their structures.
    These results motivated that properties of distance matrix of a graph have been investigated.  There are some   excellent surveys  (see \cite{aouchiche2014, bapat2014, stevanovic2012}) on this topic.

Let $G$ be a connected graph of order $n$ with vertex set $V(G)$ and edge set $E(G)$. Thus the distance matrix of $G$ is defined as $D(G)=(d_{uv})_{n\times n}$, where $d_{uv}$ is the distance between vertices $u$ and $v$ in graph $G$. Clearly,  $D(G)$ is a symmetric matrix and  its eigenvalues are real, denoted  by $\lambda_1(G)\ge \lambda_2(G)\ge \cdots\ge \lambda_n(G)$. Moreover, $\lambda_1(G)$ is called the {\it distance spectral radius} of $G$ and is denoted by $\lambda(G)$.
   Further, Indulal, Gutman and Vijaykumar \cite{indulal} in 2009 introduced  distance energy which is defined as
$$E_D(G)=\sum_{i=1}^n|\lambda_i(G)|.$$
 They obtained some sharp bounds for the distance spectral radius and D-energy of graphs with diameter 2.
 Andeli\'{c}, Koledin and  Zoran Stani\'{c} in \cite{andelic2017} obtained the exact value of  the distance energy of several types of graphs.
 Recently, Varghesea, So and Vijayakumarc \cite{varghese2018} studied how the distance energy of complete bipartite graphs changes when an edge is deleted.  Vaidya and Popat \cite{vaidya2017}  studied the distance energy  of two particular graph compositions. Recently, Stevanov\'{i}c \cite{stevanovic2018} gave more results on the graph composition.

A {\it clique tree} is a graph whose blocks are cliques (for example see \cite{blair1993}), a {\it clique path} $\mathbb{P}_{n_1,n_2,\cdots,n_k}$ is the graph obtained from the path $P_{k+1}=v_1\cdots v_{k+1}$ with order $k+1$ replacing $v_iv_{i+1}$  by $K_{n_i}$ such that
$$V(K_{n_i})\cap V(K_{n_{i+1}})=v_{i+1},~i=1,2,\cdots,k-1~\mbox{ and } V(K_{n_i})\cap V(K_{n_{j}})=\emptyset\mbox{ if }i\neq j.$$  Let $n_+(D(G)),n_-(D(G))$ denote the the number of positive and negative eigenvaues of $D(G)$, respectively. Lin, Liu and Lu \cite{Linl2015} proved  that
 \begin{theorem}\label{InT}\cite{Linl2015}
 Let $G$ be a clique tree with order $n$. Then $n_+(D(G))=1$ and $n_-(D(G))=n-1$.
 \end{theorem}
Furthermore, they proposed the following conjecture.
\begin{conjecture}\cite{Linl2015}\label{con}
Among all clique trees with cliques $K_{n_1} , ..., K_{n_k}$, the graph attains the
maximum distance energy is $\mathbb{P}_{\lfloor \frac{n-k+3}{2}\rfloor,2,\cdots,2,\lceil \frac{n-k+3}{2}\rceil}$.
\end{conjecture}
From theorem~\ref{InT}, the distance energy of a clique tree $G$ satisfies
 \begin{equation}\label{Eq1}E_D(G)=\sum_{i=1}^n|\lambda_i(G)|=\lambda_1(G)-\sum_{i=2}^n\lambda_i(G)=2\lambda_1(G),\end{equation}
 the last equation is from that the trace of $D(G)$ is $\sum_{i=1}^n\lambda_i(G)=0$. Thus it is sufficient to consider maximum spectral radius of clique trees  instead of maximum distance energy of clique trees.

   In addition,    Zhang \cite{zhang2019} studied the relation between the inertia  and the distance energy  for the line graph of unicyclic graphs.
               Moreover,  Consonni and Todeschini \cite{Consonni2008} investigated the distance energy in terms of some invariants. Drury and Lin \cite{Drury2015} characterized all connected graphs with distance energy in $[2n-2,2n]$.
   On the distance spectral radius of graphs, Liu \cite{Liu} investigated the graphs with minimal distance spectral radius among the graphs with fixed vertex connectivity, matching number and chromatic number. Ili\'{c} \cite{Ilic} characterized the minimal spectral radius among trees given matching number. Bose, Nath and Paul \cite{Bose} studied the connected graph with minimal distance spectral radius among all graphs fixed number of pendent vertices. Zhang \cite{Zhang} explored the graph with minimum distance spectral radius among all connected graphs with given diameter. Lin and Shu \cite{LinS2013} gave some sharp lower and upper bounds for the distance spectral radius and characterized the graphs with the spectral radius attaining the bounds. Lin and Feng \cite{Lin2015} characterized the connected graphs with minimal or maximal distance spectral radius among all graphs with fixed independence number.

     The main purpose of this paper is to gave a positive answer of conjecture~\ref{con}.

\section{Proof of Conjecture~\ref{con}}
Let $A$ be a symmetric matrix, $Y=\{1,2,\cdots,n\}$ be the index set of rows and columns of matrix $A$ and $\mathbb{1}_n$ be the all ones vector with $n$ elements, written by $\mathbb{1}$ if no ambiguity. Suppose $\{Y_1,Y_2,\cdots,Y_m\}$ is a partition of the set $Y$, the matrix $B=(b_{ij})$ is called {\it quotient matrix} of $A$ corresponding to the partittion, where $A_{i,j}$ is a submatrix of $A$ corresponding to the row index and column index $Y_i$ and $Y_j$ respectively and $b_{ij}$ is the average row sum of matrix $A_{i,j}$. The $n\times m$ matrix $S=(s_{ij})$ is called {\it characteristic matrix}, where $s_{ij}=1$ if $i\in Y_{j}$ and $s_{ij}=0$ if $i\notin Y_{j}$ . The partition is called {\it equitable partition} if $A_{i,j}\mathbb{1}=b_{ij}\mathbb{1}$. For more details, the readers can refer to \cite{Brouwer}.
\begin{lemma}\label{Eq-P}\cite{Brouwer} Let $B$ be the quotient matrix of $A$ corresponding to an equitable partition. If $v$ is an eigenvector of $B$ for an eigenvalue $\mu$, then $Sv$ is an eigenvector of $A$ for the same eigenvalues $\mu$, where $S$ is the characteristic matrix corresponding to the equitable partition.
\end{lemma}

\begin{corollary}\label{cor1}
$$\lambda({\mathbb{P}_{n_1+1,2,\cdots,2,n_2+1}})=\lambda(B(\mathbb{P}_{n_1+1,2,\cdots,2,n_2+1})),$$
where
\begin{eqnarray*}
&&B(\mathbb{P}_{n_1+1,2,\cdots,2,n_2+1})=\left(\begin{matrix}
n_1-1&u_{k-1}^T&kn_2\cr
n_1u_{k-1} &D(P_{k-1}) &n_2w_{k-1} \cr
kn_1&w_{k-1}^T & n_2-1
\end{matrix}\right),~\mbox{and}\\
&&u_k=(1,2,\cdots,k)^T,~~w_k=Qu_k,
\end{eqnarray*}
$Q$ is a permutation matrix, the $(i,k+1-i)$-element of which is one, $1\le i\le k$.
\end{corollary}
\vspace{0.5cm}
$$
\begin{tikzpicture}[scale=0.4]\label{fig3}
\coordinate[label=$v_2$] (I) at (5cm,3.5cm);
\coordinate[label=$v_3$] (I) at (7.5cm,3.5cm);
\coordinate[label=$v_4$] (I) at (10cm,3.5cm);
\coordinate[label=$v_{k-1}$] (I) at (17.5cm,3.5cm);
\coordinate[label=$v_{k}$] (I) at (20.5cm,3.5cm);
\coordinate[label=$K_{n_1+1}$] (I) at (2.5cm,3.8cm);
\coordinate[label=$K_{n_2+1}$] (I) at (23.5cm,3.8cm);
\coordinate[label=Fig.1 $\mathbb{P}_{n_1+1,2,\cdots,2,n_2+1}$] (I) at (13cm,0cm);
\fill[black]
(5cm,5cm) circle(2mm)
(21cm,5cm) circle(2mm)
(10cm,5cm) circle(2mm)
(7.5cm,5cm) circle(2mm)
(17.5cm,5cm) circle(2mm);
\draw[-]
(5cm,5cm)--(10cm,5cm)
(17.5cm,5cm)--(21cm,5cm)
(1cm,2cm)--(1cm,7cm)--(5cm,5cm)--cycle
(21cm,5cm)--(25cm,2cm)--(25cm,7cm)--cycle;
\draw[-,dashed]
(10cm,5cm)--(17cm,5cm);
\end{tikzpicture}$$
\begin{proof}
Let $X_1,X_2,\cdots,X_k,X_{k+1}$ be a partition of $V(\mathbb{P}_{n_1+1,2,\cdots,2,n_2+1})$ with  $X_1=V(K_{n_1+1})\setminus \{v_2\}$, $X_i=\{v_{i}\},i=2,3,\cdots,k$ and $X_{k+1}=V(K_{n_2+1})\setminus\{v_{k}\}$. Then
\begin{displaymath}
D(\mathbb{P}_{n_1+1,2,\cdots,2,n_2+1})=\left(\begin{matrix}
(J-I)_{n_1\times n_1}&\mathbb{1}u_{k-1}^T&kJ_{n_1\times n_2}\cr
u_{k-1}\mathbb{1}^T &D(P_{k-1}) &w_{k-1}\mathbb{1}^T \cr
kJ_{n_2\times n_1}&\mathbb{1}w_{k-1}^T & (J-I)_{n_2\times n_2}
\end{matrix}\right).
\end{displaymath}
The quotient matrix of $D(\mathbb{P}_{n_1+1,2,\cdots,2,n_2+1})$ corresponding to $X_1,X_2,\cdots,X_{k+1}$ is
\begin{displaymath}
B(\mathbb{P}_{n_1+1,2,\cdots,2,n_2+1})=\left(\begin{matrix}
n_1-1&u_{k-1}^T&kn_2\cr
n_1u_{k-1} &D(P_{k-1}) &n_2w_{k-1} \cr
kn_1&w_{k-1}^T & n_2-1
\end{matrix}\right).
\end{displaymath}
 Moreover, $\{X_1,X_2,\cdots,X_{k+1}\}$ is an equitable partition. By lemma~\ref{Eq-P}, $\lambda({\mathbb{P}_{n_1+1,2,\cdots,2,n_2+1}})=\lambda(B(\mathbb{P}_{n_1+1,2,\cdots,2,n_2+1}))$.
\end{proof}

Now  we consider the  characteristic polynomial of the matrix $B(\mathbb{P}_{n_1+1,2,\cdots,2,n_2+1})$.
\begin{lemma}\label{Lem1}
If $x >\lambda(P_{k-1})$ with  $k>2$, then\\
(1) $u_{k-1}^{T}[x I-D(P_{k-1})]^{-1}w_{k-1}=w_{k-1}^T[x I-D(P_{k-1})]^{-1}u_{k-1}>0$;\\
(2) $u_{k-1}^T[x I-D(P_{k-1})]^{-1}u_{k-1}=w_{k-1}^T[x I-D(P_{k-1})]^{-1}w_{k-1}>0$;\\
(3) $||u_{k-1}-w_{k-1}||^2=\frac{(k-1)k(k-2)}{3}$;\\
(4) $u_{k-1}^T[x I-D(P_{k-1})]^{-1}(u_{k-1}-w_{k-1})=\frac{1}{2}(u_{k-1}-w_{k-1})^T[x I-D(P_{k-1})]^{-1}(u_{k-1}-w_{k-1})$.
\end{lemma}
\begin{proof}
(1) Since $x>\lambda(D(P_{k-1}))$ and $D(P_{k-1})$ is a nonnegative symmetric matrix, we have
\begin{equation}\label{E1}
[x I-D(P_{k-1})]^{-1}=\frac{1}{x}\sum_{j=0}^{\infty}\left(\frac{D(P_{k-1})}{x}\right)^j,
\end{equation}
which  is  also a nonnegative symmetric matrix. Then
$$u_{k-1}^T[x I-D(P_{k-1})]^{-1}w_{k-1}=w_{k-1}^T[x I-D(P_{k-1})]^{-1}u_{k-1}>0.$$
(2).  By equation~(\ref{E1}),
$$u_{k-1}^T[x I-D(P_{k-1})]^{-1}u_{k-1}>0 \ \mbox {and} ~w_{k-1}^T[x I-D(P_{k-1})]^{-1}w_{k-1}>0.$$
Note that $w_{k-1}=Qu_{k-1}$, $ Q^TQ=I$ and $Q^TD(P_{k-1})Q=D(P_{k-1})$, we have
\begin{eqnarray*}
w_{k-1}^T[x I-D(P_{k-1})]^{-1}w_{k-1}&=&u_{k-1}^TQ^T[x I-D(P_{k-1})]^{-1}Qu_{k-1}\\
&=&u_{k-1}^T\frac{Q^T}{x}\sum_{j=0}^{\infty}\left(\frac{D(P_{k-1})}{x}\right)^jQu_{k-1}\\
&=&u_{k-1}^T\frac{1}{x}\sum_{j=0}^{\infty}\left(\frac{Q^TD(P_{k-1})Q}{x}\right)^ju_{k-1}\\
&=&u_{k-1}^T\frac{1}{x}\sum_{j=0}^{\infty}\left(\frac{D(P_{k-1})}{x}\right)^ju_{k-1}\\
&=&u_{k-1}^T[x I-D(P_{k-1})]^{-1}u_{k-1}.
\end{eqnarray*}
(3). Divide into the following two cases to prove the assertion.\\
{\bf Case 1}: If $k$ is odd and $u_{k-1}-w_{k-1}=(2-k,\cdots,-3,-1,1,3,\cdots,k-2)^T$, then
\begin{eqnarray*}
||u_{k-1}-w_{k-1}||&=&\sqrt{2(1^2+3^2+\cdots+(k-2)^2)}\\
&=&\sqrt{2\left(\frac{(k-1)k(2k-1)}{6}-2^2+4^2+\cdots+(k-1)^2\right)}\\
&=&\sqrt{2\left(\frac{(k-1)k(2k-1)}{6}-4(1^2+2^2+\cdots+(\frac{k-1}{2})^2)\right)}\\
&=&\sqrt{2\left(\frac{(k-1)k(2k-1)}{6}-\frac{(k-1)k(k+1)}{6})\right)}\\
&=&\sqrt{\frac{(k-1)k(k-2)}{3}}.
\end{eqnarray*}
{\bf Case 2}: If $k$ is even and $u_{k-1}-w_{k-1}=(2-k,\cdots,-2,0,2,4,\cdots,k-2)^T$, then
\begin{eqnarray*}
||u_{k-1}-w_{k-1}||&=&\sqrt{2\left(2^2+4^2+\cdots+(k-2)^2\right)}\\
&=&\sqrt{2\left(4(1^2+2^2+\cdots+(\frac{k-2}{2})^2)\right)}\\
&=&\sqrt{\frac{(k-1)k(k-2)}{3}}.
\end{eqnarray*}
(4). Since
\begin{eqnarray*}
&&(u_{k-1}-w_{k-1})^T[x I-D(P_{k-1})]^{-1}(u_{k-1}-w_{k-1})\\
&=&u_{k-1}^T[x I-D(P_{k-1})]^{-1}(u_{k-1}-w_{k-1})-w_{k-1}^T[x I-D(P_{k-1})]^{-1}(u_{k-1}-w_{k-1})\\
&=&u_{k-1}^T[x I-D(P_{k-1})]^{-1}u_{k-1}-u_{k-1}^T[x I-D(P_{k-1})]^{-1}w_{k-1}-w_{k-1}^T[x I-D(P_{k-1})]^{-1}u_{k-1}\\
&&+w_{k-1}^T[x I-D(P_{k-1})]^{-1}w_{k-1},
\end{eqnarray*}
then
\begin{eqnarray*}
&&(u_{k-1}-w_{k-1})^T[x I-D(P_{k-1})]^{-1}(u_{k-1}-w_{k-1})=2u_{k-1}^T[x I-D(P_{k-1})]^{-1}(u_{k-1}-w_{k-1})
\end{eqnarray*}
which is from (1) and (2). So
$$u_{k-1}^T[x I-D(P_{k-1})]^{-1}(u_{k-1}-w_{k-1})=\frac{1}{2}(u_{k-1}-w_{k-1})^T[x I-D(P_{k-1})]^{-1}(u_{k-1}-w_{k-1}).$$
\end{proof}
\begin{lemma}\label{Lem2}
If $x>\lambda({P_{k-1}})$, then
\begin{eqnarray*}
&&det (x I -B(\mathbb{P}_{n_1+1,2,\cdots,2,n_2+1}))\\
&=&[ (x+1)^2-(n_1+n_2)(1+\alpha_{k-1} )(x+1)+n_1n_2 (1+\alpha_{k-1}+k+\beta_{k-1} )(1-k+\gamma_{k-1})]\\
&&\cdot det(x I-D(P_{k-1})),
\end{eqnarray*}
where $$\alpha_{k-1}=u_{k-1}^T(x I-D(P_{k-1}))^{-1}u_{k-1},~\beta_{k-1}=u_{k-1}^T(x I-D(P_{k-1}))^{-1}w_{k-1},$$
$$\gamma_{k-1}=\frac{1}{2}(u_{k-1}-w_{k-1})^T(x I-D(P_{k-1}))^{-1}(u_{k-1}-w_{k-1}).$$
\end{lemma}
\begin{proof}
Since $x>\lambda({P_{k-1}})$, then $x I-D(P_{k-1})$ is an invertible matrix. By lemma~\ref{Lem1},
\begin{eqnarray*}
&&det (x I -B(\mathbb{P}_{n_1+1,2,\cdots,2,n_2+1}))\\
&=&det \left(\begin{matrix}
x -(n_1-1)&-u_{k-1}^T&-kn_2\cr
-n_1u_{k-1} &x I-D(P_{k-1}) &-n_2w_{k-1} \cr
-kn_1&-w_{k-1}^T &x-( n_2-1)
\end{matrix}\right)\\
&=&det \left(\begin{matrix}
x -(n_1-1)-n_1\alpha_{k-1}&0&-kn_2-n_2\beta_{k-1}\cr
-n_1u_{k-1} &x I-D(P_{k-1}) &-n_2w_{k-1} \cr
-kn_1-n_1\beta_{k-1}&0&x-( n_2-1)-n_2\alpha_{k-1}
\end{matrix}\right)\\
&&=det \left(\begin{matrix}
x -(n_1-1)-n_1\alpha_{k-1}&0&-kn_2-n_2\beta_{k-1}\cr
0 &x I-D(P_{k-1}) &0\cr
-kn_1-n_1\beta_{k-1}&0&x-( n_2-1)-n_2\alpha_{k-1}
\end{matrix}\right)\\
&=&[ (x+1)^2-(n_1+n_2)(1+\alpha_{k-1} )(x+1)+n_1n_2 (1+\alpha_{k-1} )^2-n_1n_2 (k+\beta_{k-1} )^2]\\
&&\cdot det(x I-D(P_{k-1})).\\
&=&[ (x+1)^2-(n_1+n_2)(1+\alpha_{k-1} )(x+1)+n_1n_2 (1+\alpha_{k-1}+k+\beta_{k-1} )\cdot\\
&&(1-k+\alpha_{k-1}-\beta_{k-1})]\cdot det(x I-D(P_{k-1})).
\end{eqnarray*}
Moreover,
\begin{eqnarray*}
\alpha_{k-1}-\beta_{k-1}&=&u_{k-1}^T(x I-D(P_{k-1}))^{-1}(u_{k-1}-w_{k-1})\\
&=&\frac{1}{2}(u_{k-1}-w_{k-1})^T(x I-D(P_{k-1}))^{-1}(u_{k-1}-w_{k-1})=\gamma_{k-1}. \end{eqnarray*}
Then
\begin{eqnarray*}
&&det (x I -B(\mathbb{P}_{n_1+1,2,\cdots,2,n_2+1}))\\
&=&[ (x+1)^2-(n_1+n_2)(1+\alpha_{k-1} )(x+1)+n_1n_2 (1+\alpha_{k-1}+k+\beta_{k-1} )(1-k+\gamma_{k-1})]\\
&&\cdot det(x I-D(P_{k-1})).
\end{eqnarray*}
\end{proof}

\begin{lemma}\label{Lem3}
Let  $\mathbb{P}_{n_1+1,2,\cdots,2,n_2+1}$ be a clique path with $k$ cliques, where $n_1$ and $n_2$  are nonnegative integers. Then
$$\lambda(\mathbb{P}_{n_1+1,2,\cdots,2,n_2+1})\ge \lambda(P_{k})> \frac{(k-1)(k+1)}{3}.$$
\end{lemma}
\begin{proof}
Since $D(P_{k})$ is a principal submatrix of $D({\mathbb{P}_{n_1+1,2,\cdots,2,n_2+1}})$, we have
$\lambda(\mathbb{P}_{n_1+1,2,\cdots,2,n_2+1})\ge \lambda(P_{k})>\frac{\mathbb{1}_k^TD(P_{k})\mathbb{1}_k}{k}=\frac{(k-1)(k+1)}{3}.$
\end{proof}
\begin{lemma}\label{L2.7}\cite{Linl2015}
Among all clique trees with cliques $K_{n_1},...,K_{n_k},$ the graph attains the maximum distance spectral radius is $P_{n_1'+1,2,...,2,n_k'+1}$
\end{lemma}
Now  we  are ready to prove conjecture~\ref{con}.

\begin{proof}
By theorem~\ref{InT}, it is sufficient to prove that $\mathbb{P}_{\lfloor \frac{n-k+3}{2}\rfloor,2,\cdots,2,\lceil \frac{n-k+3}{2}\rceil}$ has the maximum spectral radius in the clique tree. By lemma~\ref{L2.7}, the graph which attains the maximum distance spectral radius is the clique path $\mathbb{P}_{n_1+1,2,\cdots,2,n_2+1}$. If $n_1,n_2,n_1',n_2'$ are positive integers such that
\begin{equation}\label{Ceq}\max\{n_1,n_2\}<\max\{n_1',n_2'\} ~~\mbox{and}~~ n_1+n_2=n_1'+n_2'=n-1-k,\end{equation}
which implies that $n_1n_2>n_1'n_2'$, we will show that $$\lambda(\mathbb{P}_{n_1+1,2,\cdots,2,n_2+1})>\lambda(\mathbb{P}_{n_1'+1,2,\cdots,2,n_2'+1}).$$

Let $f(n_1,n_2,x)=det (x I -B(\mathbb{P}_{n_1+1,2,\cdots,2,n_2+1}))$ be the characteristic polynomial of matrix $B(\mathbb{P}_{n_1+1,2,\cdots,2,n_2+1})$, where $B(\mathbb{P}_{n_1+1,2,\cdots,2,n_2+1})$ is a $(k+1)\times (k+1)$ matrix, similarly define  $f(n_1',n_2',x)$.
By the definition of $f(n_1,n_2,x), f(n_1',n_2',x)$ and lemma~\ref{Lem2}
\begin{eqnarray}\label{MEQ}
f(n_1,n_2,x)-f(n_1',n_2',x)&=&(n_1n_2-n_1'n_2')(1+\alpha_{k-1}+k+\beta_{k-1} )(1-k+\gamma_{k-1})]\nonumber\\
&&\cdot det(x I-D(P_{k-1})).
\end{eqnarray}
Suppose $x\ge \lambda(P_{k})>\lambda(P_{k-1})$,  by lemma~\ref{Lem1},
\begin{equation}\label{Meq1}\alpha_{k-1},\beta_{k-1}>0~\mbox{and}~det(x I-D(P_{k-1}))>0.\end{equation}
By the Perron-Frobenius theorem, there is an unique unit positive eigenvector $X=(x_1,x_2,\cdots,x_{k-1})^T$ of $D(P_{k-1})$ corresponding to $\lambda(P_{k-1})$. Then
$$\lambda(P_{k-1})=X^TQ^T D(P_{k-1})QX=X^TD(P_{k-1})X,$$
which implies that $QX$ is also an unit positive eigenvector of $D(P_{k-1})$ corresponding to $\lambda(P_{k-1})$, where $Q$ is a $(k-1)\times (k-1)$ permutation matrix  whose $(i,k-i)$-element is one for $1\le i\le k-1$. Thus $X=QX$, that is, $x_i=x_{k-i}$, $1\le i\le k-1$. Hence
$$X^T(u_{k-1}-w_{k-1})=0. $$
Furthermore, by Courant-Fischer theorem (refer to theorem 4.2.6 in \cite{Horn}),
 \begin{eqnarray}\label{Meq4}
\gamma_{k-1}&=&\frac{1}{2}(u_{k-1}-w_{k-1})^T[x I-D(P_{k-1})]^{-1}(u_{k-1}-w_{k-1})\nonumber\\
&\le& \frac{1}{2}||u_{k-1}-w_{k-1}||^2\frac{1}{x -\lambda_2(D(P_{k-1}))}.
\end{eqnarray}
By $x\ge \lambda(P_{k})>\lambda(P_{k-1})$ and lemma~\ref{Lem3}, we have
\begin{equation}\label{Eq3}
x\ge  \lambda(P_{k})> \frac{(k-1)(k+1)}{3}.
\end{equation}
By theorem~\ref{InT} and lemma~\ref{Lem1},
\begin{equation}\label{Meq2}\lambda_2(D(P_{k-1}))<0 \mbox{ and }||u_{k-1}-w_{k-1}||^2=\frac{(k-1)k(k-2)}{3}.\end{equation}
Inequalities $(\ref{Meq4})$, $(\ref{Eq3})$ and $(\ref{Meq2})$ implies that
\begin{eqnarray}\label{Meq5}
\gamma_{k-1}\le\frac{(k-1)k(k-2)}{6}\frac{1}{x }<\frac{(k-1)k(k-2)}{6}\frac{1}{\frac{(k-1)(k+1)}{3} }<k-1.
\end{eqnarray}
Take inequalities $(\ref{Meq1})$, $(\ref{Meq5})$ and $n_1n_2>n_1'n_2'$ into $(\ref{MEQ})$,
$$f(n_1,n_2,x)-f(n_1',n_2',x)>0.$$
Which implies that $f(n_1,n_2,x)-f(n_1',n_2',x)>0$ for $x\ge \lambda(P_{k})$. Since $\lambda(\mathbb{P}_{n_1+1,2,\cdots,2,n_2+1})\ge \lambda(P_{k})$ and $\lambda(\mathbb{P}_{n_1'+1,2,\cdots,2,n_2'+1})\ge \lambda(P_{k})$, then
$$\lambda(\mathbb{P}_{n_1+1,2,\cdots,2,n_2+1})>\lambda(\mathbb{P}_{n_1'+1,2,\cdots,2,n_2'+1}).$$
Since $\lfloor \frac{n-k+1}{2}\rfloor,\lceil \frac{n-k+1}{2}\rceil$ are the only numbers satisfying that, for every positive integers $n_1',n_2'$ and $n_1'+n_2'=n+1-k $ such that
$$\max\{\lfloor \frac{n-k+1}{2}\rfloor,\lceil \frac{n-k+1}{2}\rceil\}\le \max\{n_1',n_2'\}\mbox{ and }$$
$$\lfloor \frac{n-k+1}{2}\rfloor+\lceil \frac{n-k+1}{2}\rceil=n_1'+n_2'=n+1-k. $$
Thus $\mathbb{P}_{\lfloor \frac{n-k+1}{2}+1\rfloor,2,\cdots,2,\lceil \frac{n-k+1}{2}\rceil}+1=\mathbb{P}_{\lfloor \frac{n-k+3}{2}\rfloor,2,\cdots,2,\lceil \frac{n-k+3}{2}\rceil}$ has the maximum spectral radius in the clique tree.
\end{proof}
\begin{center}
{\bf Acknowledgments}
\end{center}
The authors are grateful to the referees for their valuable comments and suggestions which lead to a great improvement of this paper.

\end{document}